\documentclass[12pt,twoside,reqno,english]{amsart}
\usepackage[utf8]{inputenc}
\usepackage{amsmath,url}
\usepackage{amsthm}
\usepackage{todonotes}
\usepackage[shortlabels]{enumitem}
\usepackage{amsfonts}
\usepackage{graphicx}
\graphicspath{ {./images/} }
\usepackage{breqn}

\newcounter{thecounter}
\numberwithin{thecounter}{section}
\theoremstyle{theorem}
\newtheorem{lemma}[thecounter]{Lemma}
\newtheorem{theorem}[thecounter]{Theorem}

\newtheorem{prob}[thecounter]{Problem}

\newtheorem{coro}[thecounter]{Corollary}

\theoremstyle{definition}
\newtheorem{definition}[thecounter]{Definition}
\newtheorem{remark}[thecounter]{Remark}

\DeclareMathOperator{\lk}{lk}
\DeclareMathOperator{\Ast}{ast}

\newcommand{\diam}{\mathrm{diam}}
\newcommand{\s}{\Sigma}
\newcommand{\set}[2]{\left\{#1\ \middle\vert\ #2\right\}}

\begin{document}
\title{On flag-no-square $4$-manifolds}
\author{Daniel Kalmanovich}
\address{School of Computer Science\\ Ariel University\\ Ariel~4070000\\ Israel.}
\email{danielk@ariel.ac.il}
\author{Eran Nevo}
\address{Einstein Institute of Mathematics\\
 Hebrew University\\ Jerusalem~91904\\ Israel.}
\email{nevo@math.huji.ac.il}
\thanks{E.N. was partially supported by the Israel Science Foundation grant ISF-2480/20 and by ISF-BSF joint grant 2016288.}
\author{Gangotryi Sorcar}
\address{Department of mathematical sciences\\ University of Delaware\\ Newark, DE~19716, USA.}
\email{gangotryi@gmail.com}
%\date{today}

\begin{abstract}
Which $4$-manifolds admit a flag-no-square (fns) triangulation? We introduce the ``star-connected-sum" operation on such triangulations, which preserves the fns property, from which we derive new constructions of fns $4$-manifolds. In particular, we show the following: (i) there exist non-aspherical fns $4$-manifolds, answering in the negative a question by Przytycki and Swiatkowski; (ii) for every large enough integer $k$ there exists a fns $4$-manifold $M_{2k}$ of Euler characteristic $2k$, and further, (iii) $M_{2k}$ admits a super-exponential number (in $k$) of fns triangulations - at least $2^{\Omega(k \log k)}$ and at most $2^{O(k^{1.5} \log k)}$.
\end{abstract}

\maketitle

\section{Introduction}
A simplicial complex $\Delta$ is \emph{flag} if its faces are exactly the cliques of its $1$-skeleton. If additionally, $\Delta$ has no induced cycles of length at most $k-1$ then $\Delta$ is \emph{$k$-large}; where $5$-large is also called \emph{flag-no-square}, or fns for short.

It is known that for every simplicial complex of dimension $d\le 3$ there exists a fns complex homeomorphic to it, see Dranishnikov~\cite{Dran99} for $d\le 2$ and Przytycki and Swiatkowski~\cite{PS09} for $d=3$, while the $4$-sphere has no fns triangulation, due to Januskiewicz and Swiatkowski \cite{JS03}; thus, a $d$-manifold has a fns triangulation if $d<4$ and has no fns triangulation if $d>4$ (we consider only compact manifolds without boundary, called simply \emph{manifolds}).   The following question is open:
\begin{prob}
Which triangulable $4$-manifolds admit a flag-no-square triangulation?
\end{prob}
Some lower bounds on the Euler characteristic  $\chi(M)$ of $4$-manifolds $M$ admitting a fns triangulation are known: Kopcz\'{y}nski, Pak and Przytycki~\cite{KPP09} showed that $\chi(\Delta)\ge f_0(\Delta)$ if $\Delta$ is a fns triangulation of a $4$-manifold, where $f_0(\Delta)$ is the number of vertices of $\Delta$. Hence, e.g. $\chi(M)\ge 18$ by passing to links, 
as one can show that the icosahedron, which has $12$ vertices, minimizes the number of vertices among fns triangulations of the $2$-sphere (plus little extra work to reach the $18$ bound), and likely $\chi(M)\ge 603$, which would follow if the boundary of the $600$-cell minimizes the number of vertices among fns triangulations of the $3$-sphere.
On the other hand, very limited constructions of fns $4$-manifolds are known, described in~\cite[4.4(2)]{PS09}: they are all quotients of the regular simplicial tesselation
of the hyperbolic $4$-space with all vertex links isomorphic to the boundary complex of the $600$–cell, by an appropriate subgroup of its (Coxeter) automorphism group; in particular, the resulting fns manifolds are aspherical.

The following two lemmas allow us to construct new examples of fns $4$-manifolds. Suppose $N$ and $M$ are 
two disjoint 
simplicial complexes,
%(they need not be distinct)
 and suppose there is a combinatorial isomorphism between the vertex links $\lk_v(N)$ and $\lk_u(M)$ with the isomorphism being induced by some bijection on their vertices $\phi:V(\lk_v(N))\rightarrow V(\lk_u(M))$. We will use the same notation for the combinatorial isomorphism between the vertex links and write it as $\phi:\lk_v(N)\rightarrow \lk_u(M)$.
Also recall that for a vertex $v\in M$, the antistar $\Ast_v(M)$ is the subcomplex consisting of all faces of $M$ not containing $v$.

\begin{definition}
    The \textbf{star connected sum} $N\#_{\phi}M$ is the simplicial complex obtained by gluing the antistars $\Ast_v(N)$ and  $\Ast_u(M)$ according to $\phi$, namely we take the union of the antistars and identify $w$ with $\phi(w)$ for all $w\in V(\lk_v(N))$.
\end{definition}

\begin{lemma}\label{lem:star-connected-sum}
Assume $N$ and $M$ are disjoint fns orientable connected $4$-manifolds, $v\in V(N)$, $u\in V(M)$ and $\phi: \lk_v(N) \rightarrow \lk_u(M)$ a combinatorial isomorphism that reverses orientation. Then 

(i) $N\#_{\phi}M$ is an orientable fns $4$-manifold, homeomorphic to the connected sum 
$N\# M$, 
and

(ii) $\chi(N\#_{\phi}M)=\chi(N)+\chi(M)-2$.
\end{lemma}

Lemma~\ref{lem:star-connected-sum} implies the following corollary.
\begin{coro}\label{cor:nonaspherical}
There exist non-aspherical $4$-manifolds admitting a fns triangulation.
\end{coro}
This settles in the negative Question 5.8(1) in~\cite{PS09}. 

Next we use a handle-type construction, similar to the star connected sum:
%, to provide a new simplicial complex from a 
given a connected simplicial manifold $N$ admitting two vertices $v$ and $u$ of graph distance at least five with %disjoint 
isomorphic vertex links $\lk_v(N)$ and $\lk_u(N)$, denote by 
$\phi:V(\lk_v(N))\rightarrow V(\lk_u(N))$ a bijection that induces an isomorphism of the links.
Under these conditions, 

\begin{definition} 
%Let $\phi:V(\lk_v(N))\rightarrow V(\lk_u(N))$ be an isomorphism of simplicial complexes.
    The \textbf{star handled} $N$, denoted by $h_{\phi}(N)$, is the simplicial complex\footnote{The condition on the graph distance $d(u,v)\ge 5$ guarantees that $h_{\phi}(N)$ is indeed a simplicial complex.} obtained from the subcomplex $(N-v)-u$ of $N$ by identifying the links $\lk_v(N)$ and  $\lk_u(N)$ according to $\phi$, namely we take the quotient of $(N-v)-u$ by identifying $w$ with $\phi(w)$ for all $w\in V(\lk_v(N))$. 
\end{definition}
Note that topologically $h_{\phi}(N)$ is homeomorphic to $N$ with a (hollow) $1$-handle attached. The following Lemma follows.

%\begin{definition}
%    Given a $4$-manifold $N$ and homeomorphisms $\psi_i:\mathbb{D}^4\rightarrow N$ for $i=0,1$ such that they have disjoint images and $\psi_0$ preserves orientation and $\psi_1$ reverses orientation. Consider the manifold obtained by taking the disjoint union of $N-(\psi_0(\mathbb{D}^4)\cup\psi_1(\mathbb{D}^4))$ and $\mathbb{S}^3\times [0,1]$ and identifying $(x,i)$ with $\psi_i(x)$ for each $x$ in $\mathbb{S}^3\subset \mathbb{D}^4$ and $i=0,1$. We will define this manifold to be $N$ with a \textbf{hollow $1$-handle} attached.
%\end{definition}

\begin{lemma}\label{lem:handle-connected-sum}
Let $N$ be a fns orientable connected $4$-manifold and $v,u\in V(N)$ of distance at least seven in the graph metric of the $1$-skeleton of $N$,  and $\phi: \lk_v(N) \rightarrow \lk_u(N)$ a combinatorial isomorphism that reverses orientation. Then 

(i) $h_{\phi}(N)$ is a fns $4$-manifold, 
homeomorphic to  $N$ with a hollow $1$-handle attached,
and

(ii) $\chi(h_{\phi}(N))=\chi(N)-2$.
\end{lemma}

Let us elaborate a bit on the construction scheme in~\cite[Remark 4.4(2)]{PS09}: it uses 
  the $5$-large tessellation $T$ of the hyperbolic $4$-space $\mathbb{H}^4$ by regular simplices where each vertex link is $X_{600}$, the boundary of the $600$-cell. The automorphism group of $T$ is the Coxeter group $G=[5,3,3,3]$. By Selberg's lemma \cite{s60}, any finitely generated linear group over a field of characteristic zero has a finite index torsion free subgroup, hence $G$ has a finite index torsion free subgroup $H$. Also, any Coxeter group is residually finite by Malcev's theorem \cite{m40} and therefore so is $H$ since subgroups inherit residual finiteness. As $H$ is residually finite, it has a subgroup $K$ such that the minimal displacement of the $K$-action on $T$ is arbitrarily large, in particular $\inf\{d(v,kv): v \in V(T), k \in K\}>5$ where $d(.,.)$ is the usual graph metric on the $1$-skeleton of $T$. Therefore, the quotient $\mathbb{H}^4/K$ is a $5$-large simplicial complex homeomorphic to a closed hyperbolic $4$-manifold.
  %~\footnote{In ~\cite[Remark 4.4(2)]{PS09}, $K$ with large injectivity radius, rather than large minimal displacement, was used.}
  One can also assume that these manifolds are orientable simply by replacing $K$ with its index $2$ subgroup containing only the orientation preserving elements of $K$. 
  
%\todo[inline]{Ganges: I got this minimal displacement argument from Osajda's ``a construction of hyperbolic Coxeter gps" paper. Not sure if it should be referenced because it is not a result he proves but a fact he quotes and uses in a slightly different context. I think the minimal displacement argument is less distracting than the injectivity radius one. Both are very similar in nature though. Eran: I find the text OK as is now}

%From pg. 2473 \cite{CM05}, we know that each torsion free subgroup of $G$ has an index that is divisible by $14400$. Moreover, note that if the index of any subgroup  $\Gamma$ is $k.14400$ then $\chi(\mathbb{H}^4/\Gamma)=k$. 

Since closed hyperbolic $4$-manifolds have even Euler characteristic, all of these fns manifolds have even Euler characteristic. The question arises: which positive integers can be realized as the Euler characteristic of some fns $4$-manifold? 

By iterative use of Lemmas~\ref{lem:star-connected-sum} and~\ref{lem:handle-connected-sum}, starting with manifolds as above obtained from subgroups with minimal displacement at least 13, we obtain Corollary~\ref{cor:highEvenXhi} -- partially answering this question, by providing new constructions.

\begin{coro}\label{cor:highEvenXhi}
% There exists a positive integer $k_0$ such that for all $k\ge k_0$ there exists a fns triangulation $\Delta$ of some  connected $4$-manifold $M_{2k}$ with $\chi(M_{2k})=2k$.

For all even $s$ large enough there exists a $4$-manifold $G_s$, admitting a fns triangulation, with $\chi(G_s)=s$.
\end{coro}

Given Corollary~\ref{cor:highEvenXhi}, it is natural to ask how many combinatorially distinct fns triangulations can a connected $4$-manifold admit. 
By the result of~\cite{KPP09} mentioned above, that $\chi(\Delta)\ge f_0(\Delta)$, indeed there are only finitely many such triangulations. We give a better estimate:
let $t(M)$ be the number of combinatorial types of fns triangulations of a given connected $4$-manifold $M$, and let 
$$t(x)=\sum_{\chi(M)=x} t(M)$$ where the summation runs over all connected $4$-manifolds of Euler characteristic $x$. 

\begin{theorem}\label{cor:count-fns}
\begin{enumerate}[(i)]
    \item For all even $k$ large enough there exists a $4$-manifold $M_k$ with $\chi(M_k)=\Theta(k)$ that has at least $k!$ combinatorially distinct fns triangulations. Thus:
    $$t(M_k)=2^{\Omega(\chi(M_k)\log \chi(M_k))}.$$
    
    \item There exists a constant $b>0$ such that for all $x$ large enough: $$t(x) < b^{x^{3/2} \log x}.$$
\end{enumerate}
\end{theorem}
The proof of item (i), given in Section~\ref{sec:symm}, is based on a 
%symmetry breaking 
technique, which may be of independent interest: there, although we know very little about the fns $4$-manifolds constructed in~\cite{PS09}, we can still take many copies of one such construction, and by carefully applying the 
%vertex-star gluing 
star-connected-sum and star-handle
operations to them, guarantee that we obtain factorially many non-isomorphic fns triangulations of the same  $4$-manifold.
The proofs of the other results are given in Section~\ref{sec:Proofs}. We conclude in Section~\ref{sec:Open} with related open problems.

\section{Proofs}\label{sec:Proofs}
First, we note that there exists an orientation reversing combinatorial isomorphism $\phi:X_{600}\rightarrow X_{600}$, as $X_{600}$ can be embedded in $\mathbb{R}^4$ in a way that if $(x,y,z,w)$ is a vertex, then so is $(x,y,z,-w)$; see~\cite[p. 247]{Cox73}). The above map induces a degree $-1$ map on the $3$-sphere $X_{600}$ (see e.g.~\cite[p. 134]{Hat02}), thus it is orientation reversing.
We will use such an orientation reversing isomorphism throughout the proofs below.
\begin{proof}[Proof of Lemma~\ref{lem:star-connected-sum}]
%\begin{enumerate}
%\item 
Note that $\Ast_v(N)$ and  $\Ast_u(M)$ are homeomorphic to $N\setminus B_N$ and $M\setminus B_M$ respectively where $B_N$ and $B_M$ are two open $4$-balls in $N$ and $M$ respectively. Also, the combinatorial isomorphism $\phi$ is a homeomorphism between $\partial B_N$ and $\partial B_M$.\footnote{Note that by choosing appropriate orientations on $N$ and $M$, for any link isomorphism $\phi$ one can assume that $\phi$ is orientation reversing.}
Hence, the star connected sum $N\#_\phi M$ is homeomorphic to the usual connected sum $N\# M$. The connected sum $N\# M$ is well-defined, that is it does not depend on the choice of the balls $B_N$ and $B_M$ or the orientation reversing homeomorphism between their boundaries $\partial B_N$ and $\partial B_M$. Hence, if different vertices are chosen in $N$ and $M$ respectively and a different orientation reversing link isomorphism is used to obtain the star connected sum, the resulting manifold will be homeomorphic to $N\#_{\phi}M$. The well-definedness of $N\#M$ is due to Quinn's Annulus Theorem in dimension $4$ (the Annulus Theorem was gradually proved earlier in all other dimensions by others) and a sketch of its proof can be found in Lee Mosher's answer to a MathOverflow question by Antoine Chambert-Loir (see~\cite{mathoverflow:LeeMosh}).

%\todo[inline]{Eran: are the [PS]-manifolds orientable? Ganges: Yes, there are orientable [PS] manifolds and I have added this point where we discuss the construction of [PS]-mflds above. Eran: If so, different gluings may give two different manifolds, and we need to be more specific in pf of the exp bound of 1.5(i). If not, we need to explain also about the non-orientable case.}

%\todo[inline]{Eran: as above, careful with the isomorphism, not every isom will do! Ganges: Any orientation reversing isom will do.}

 Note that $\Ast_v(N)$, $\Ast_u(M)$, $\lk_v(N)$, and $\lk_u(M)$ are flag-no-square since they are induced subcomplexes of the fns simplicial complexes $N$ and $M$ respectively. 
 %While gluing $\Ast_v(N)$ and $\Ast_u(M)$ by the isomorphism $\phi: \lk_v(N) \rightarrow \lk_u(M)$ we do not introduce any new simplices, therefore the resulting simplicial complex $N\#_{\phi}M$ is flag.
 Note that the union $A\cup B$ of two flag complexes $A$ and $B$, whose intersection $A\cap B$ is an \emph{induced} subcomplex in both, must be flag as well; thus $N\#_{\phi}M$ is flag.
 Now, assume by contradiction that there is an induced square, namely $4$-cycle, $abcd$ in $N\#_{\phi}M$. As $abcd \not\subseteq \Ast_u(M)$ and $abcd \not\subseteq \Ast_v(N)$, this induced square has to have exactly one vertex in $\Ast_u(M) \setminus \Ast_v(N)$ and exactly one vertex in $\Ast_v(N) \setminus \Ast_u(M)$; as shown in the following figure.
 %, where $a, c \notin \lk_u(M)$.

$$\includegraphics[scale=1]{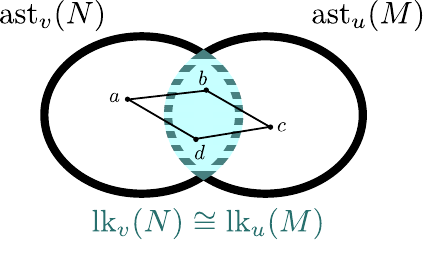}$$

So if the square has vertices $a$, $b$, $c$, and $d$, then without loss of generality, we can assume that $a \in \Ast_v(N)\setminus \lk_v(N)$, $c \in \Ast_u(M)\setminus \lk_u(M)$ and $b, d \in \lk_u(M)=\lk_v(N)$ (under their identification). 
But then $ubcd$ would be an induced square in $M$, a contradiction.
This proves item (i).

%But $b, d \in \lk_u(M)$ implies that edges $bv$ and $dv$ belong to the triangulation of $M$, giving rise to the square with vertices $a$, $b$, $v$, and $d$ in order inside the triangulation of $M$. Since we assumed this triangulation is flag-no-square, this square $abvd$ must have a diagonal. The diagonal cannot be $av$ because $a\notin \lk_u(M)$, hence the diagonal has to be $bd$.

%Now we want to argue that the edge $bd$ belongs to  $N\#_{\phi}M$. We do this by showing that the edge $bd \in \lk_u(M)$. Since $b$, $v$, and $d$ are all mutually connected by edges and since the triangulation of $M$ is flag, the $2$-simplex $bdv \in M$, which implies that the edge $bd \in \lk_u(M)$. Hence, the square $abcd$ is not an induced square in $N\#_{\phi}M$.\\

    %\item 
   %\todo[inline]{Eran: I corrected below, previously the face count was false. Agreed? Ganges: The formula $f(ast_v(N))=(f(N)-t\cdot f(\lk_v(N))$ isn't clear to me. I run into an issue because $f_0(st)=f_0(lk)+1$. Eran: the empty set is a face. Anyhow, I removed the count altogether, referring to MV instead.}
    
    Item (ii) follows directly from (i), by applying the Mayer--Vietoris sequence to our connected sum. Indeed, the Betti numbers satisfy 
    $$\beta_i(N\# M)=\beta_i(N)+\beta_i(M)$$ 
    for $i=1,2,3$, and $\beta_0(N\# M)=1=\beta_4(N\# M)$.\footnote{Alternatively, instead of computing Euler characteristic via Betti numbers, one could compute it via face numbers.}
%    Given a simplicial complex $X$, let $f_i(X)$ be the number of its $i$-dimensional faces. 
%    The $f$-polynomial of $X$ in variable $t$ is $f(X)(t)=\sum_{F\in X}t^{|F|}=\sum_i f_i(X) t^{i+1}$.
%    We know that the Euler characteristic of an $n$-dimensional simplicial complex $X$ is $\chi(X)=\Sigma_{i=0}^n (-1)^if_i(X)$, and its reduced Euler characteristic is $\Tilde{\chi}(X)=\Sigma_{i=-1}^n (-1)^if_i(X)$, where $f_{-1}(X)=1$ for the empty set in $X$. Thus, $\Tilde{\chi}(X)=-f(X)(-1)$.       
%    By definition of $N\#_{\phi}M$, we see that 
%\begin{dmath*}
%f(N\#_{\phi}M)= f(\Ast_v(N))+f(\Ast_u(M))-f(\lk_u(M)) =
%(f(N)-t\cdot f(\lk_v(N))) + (f(M)-t\cdot f(\lk_u(M))) - f(\lk_u(M)) = f(N)+f(M)-(1+2t)f(\lk_u(M)).
%\end{dmath*}  
%Passing to reduced Euler characteristics we get
%\begin{dmath*}
%\Tilde{\chi}(N\#_{\phi}M) = -f(N\#_{\phi}M)(-1)= -[f(N)(-1)+f(M)(-1)-(1-2)f(\lk_u(M))(-1)]=
%\Tilde{\chi}(N)+\Tilde{\chi}(M)+\Tilde{\chi}(\lk_u(M)).
%\end{dmath*}
%Now, $\lk_u(M)$ is homeomorphic to $\mathbb{S}^3$ so $\Tilde{\chi}(\lk_u(M))=-1$.
%Thus, $\chi(N\#_{\phi}M)-1=(\chi(N)-1)+(\chi(M)-1)-1$, that is 
%$\chi(N\#_{\phi}M)=\chi(N)+\chi(M)-2$, as claimed.    
%    $$f_0(N\#_{\phi}M) =f_0(N)-1+f_0(M)-1$$ and $f_i(N\#_{\phi}M) =f_i(N)-f_{i-1}(\lk_v(N))+f_i(M)-f_{i-1}(\lk_u(M))$ for all $i=1,2,3,4$.
%    Therefore, $\chi(N\#_{\phi}M)=\chi(N)+\chi(M)-2+\chi(\lk_v(N))+\chi(\lk_u(M))$. Since $N$ and $M$ are $4$-manifolds, $\lk_v(N)$ and $\lk_u(M)$ are homeomorphic to $\mathbb{S}^3$. Since $\chi(\mathbb{S}^3)=0$, our proof follows. 
%\end{enumerate}
\end{proof}

\begin{proof}[Proof of Corollary \ref{cor:nonaspherical}]
Let $M$ be a fns $4$-manifold from the examples in~\cite[4.4(2)]{PS09}, and let $M_1$ and $M_2$ be two disjoint copies of $M$. We consider $M_1\#_{\phi}M_2$, where $\phi:\lk_{u_1}(M_1)\rightarrow \lk_{u_2}(M_2)$ is an orientation reversing isomorphism, and $u_i$ is any vertex in $M_i$, $i=1,2$. By Lemma~\ref{lem:star-connected-sum}, we get that $X=M_1\#_{\phi}M_2$ is fns. 
Being the connected sum of two aspherical manifolds, it is known that $X$ is not aspherical; see e.g.~\cite[p. 20, Theorem 2.1]{McCullough} for a much stronger result, or a short proof by Robert Bell on MathOverflow~\cite{mathoverflow:RobertBell}.
\end{proof}

%\todo[inline]{Eran: I replaced the pf above by a reference -- where exactly in [McCullough] to cite? [OLD: Eran: The rest of the proof is not clear to me. Is there a reference to the claim: connected sum of aspherical is non-aspherical? Is there an issue with orientations, as in our Lemma? Ganges: I elaborated more, maybe it is clear now? The reference you emailed by Darryl McCullough is a much stronger result and indeed implies this corollary. We can refer to it instead of providing the proof if you want. There is no issue with orientations here.]}

%\todo[inline]{Eran: Do we know that $M$ has nontrivial fundamental group? What is this group? Ganges: If $M=\mathbb{H}^4/K$, then $\pi_1(M)=K$. Here $K$ is a suitable torsion free finite index subgroup of the Coxeter group $G=[5,3,3,3]$ and it is non-trivial.} 

%\todo[inline]{Eran: I changed the pf below, please check.}

\begin{proof}[Proof of lemma~\ref{lem:handle-connected-sum}]
%Assume that $v$ and $u$ are two vertices in $N$ such that the graph metric $2$-balls around them are disjoint. Also assume that each of these $2$-balls are homeomorphic to a $4$-disk. We will use the notation $B^0(r,x)$ for the interior of an $r$-radius ball centered at $x$. Notice that $h_\phi(N)$ is homeomorphic to $((N-B^0(2,v))-B^0(2,u))\sqcup(B(2,v)\#_\phi B(2,u))/\sim$, where $\sim$ is the identification of the boundaries of $B(2,v)$ and $B(2,u)$ by the respective identity isomorphisms. Since by lemma~\ref{lem:star-connected-sum}, the star-connect sum $B(2,v)\#_\phi B(2,u)$ is homeomorphic to $\mathbb{S}^3\times [0,1]$, we get a homeomorphism between $h_\phi(N)$ and $N$ with a hollow $1$-handle attached.
We know that $h_\phi(N)$ is homeomorphic to $N$ with a hollow $1$-handle attached.
To see that $h_\phi(N)$ is fns, note that all induced cycles in $h_\phi(N)$ formed by identifying the ends $w$ and $\phi(w)$ in an induced path $w\ldots\phi(w)$ in $N$ have length at least $5$, as the corresponding induced path $vw\ldots\phi(w)u$ in $N$ has length at least $7$ by assumption. This proves item (i). 
The Euler characteristic formula $\chi(h_{\phi}(N))=\chi(N)-2$ follows from straightforward computations using the Mayer--Vietoris sequence, proving item (ii).
\end{proof}

\begin{proof}[Proof of Corollary~\ref{cor:highEvenXhi}]
Let $M=T/K$ where $T$ is a $5$-large tessellation of $\mathbb{H}^4$ by regular $4$-simplices where all vertex links are isomorphic to the boundary of the 600-cell and 
$K$ is a torsion-free subgroup of $G=[5,3,3,3]$ of finite index with minimal displacement at least 8; 
such $T$ and $K$ exist as explained in the Introduction, and the resulted $M$ is a connected fns $4$-manifold of the construction scheme in~\cite{PS09}. 

Now, let $v$ and $u$ be diametrically apart in the $1$-skeleton $G$ of $M$, and let $P=vabv'c\ldots du'eu$ be an induced path realizing the diameter of $G$.
Let $h(M):=h_\phi(M)$ be the star handled $M$ obtained from an orientation reversing isomorphism $\phi:\lk_v(M)\longrightarrow \lk_u(M)$.
Let $\chi(M)=m$. Then $h(M)$ is fns with $\chi(h(M))=m-2$ by Lemma~\ref{lem:handle-connected-sum}. 
We use $M$ and $h(M)$ as building blocks to construct our desired manifolds $G_s$
for all $s\ge \frac{(m-6)(m-4)}{2}+2$. 

Note that in $h(M)$ the vertices $v'$ and $u'$ have disjoint links, each is isomorphic to the boundary of the 600-cell.  

Take a ``row" of $k$ copies of $h(M)$ with 
$$k(m-4)+2\leq s< (k+1)(m-4)+2,$$ 
namely, an iterated star connected sum 
that glues the link of the vertex $v'$ of the $i$th copy to the link of the vertex $u'$ of the $(i+1)$th copy, for $1\le i\le k-1$. 
Denote the resulted fns manifold by $M'$.

If equality $k(m-4)+2=s$ above holds then $\chi(M')=s$ by Lemma~\ref{lem:star-connected-sum}, and we are done. 
Otherwise, replace (the last) $\frac{s-k(m-4)-2}{2}$ copies of $h(M)$ in the row with copies of $M$, where we keep gluing the link of $v'$ in the $i$th copy (of $M$ or of $h(M)$) to the link of $u'$ in the $(i+1)$th copy in the row for all $i$ but the last one. For this to be possible we require $k\ge\frac{m-6}{2}$. Denote the resulted fns manifold by $G_s$. Again, by Lemma~\ref{lem:star-connected-sum}, $\chi(G_s)=s$. 
 
To summarize, we constructed connected fns $4$-manifolds $G_s$ with $\chi(G_s)=s$ for all even $s\ge \frac{(m-6)(m-4)}{2}+2$.

\end{proof}

\begin{proof}[Proof of Corollary~\ref{cor:count-fns}(ii)]
%    \begin{enumerate}%[(i)]
%        \item See the next section.  
        
%        \item 
We will give an upper bound on the number of graphs in a family $F$ of graphs which contains all graphs that occur as $1$-skeletons of fns $4$-manifolds with Euler characteristic $x$. Recall that the latter contains only graphs with at most $x$ vertices by~\cite{KPP09}.
         
Every fns $4$-manifold is determined by its $1$-skeleton $G$ (by flagness), and (i) $G$ has no induced $C_4$ and (ii) $G$ has largest clique size $5$. From~\cite[Theorem 1]{GyarHS02} it follows that a graph with $n$ vertices satisfying (i) and (ii) has at most $\frac{5}{\sqrt 2}n^{1.5}$ edges.

Let 
$$F=\set{G}{|V(G)|\le x, |E(G)|\le \frac{5}{\sqrt 2}|V(G)|^{1.5}}.$$ 
Then
\begin{align*}
t(x)&\le |F|= \sum_{y\le x}\sum _{i\leq \frac{5}{\sqrt 2}y^{1.5}}\binom{\binom{y}{2}}{i}< x \sum _{i\leq \frac{5}{\sqrt 2}x^{1.5}}\binom{\binom{x}{2}}{i}
\end{align*}
Bounding the RHS from above, using $x^2>\binom{x}{2}$  
and $x^2/2 \geq \frac{5}{\sqrt 2}x^{1.5}$, which holds for 
$x\ge 50$, we obtain by estimating the binomial coefficients:
\begin{align*}
t(x)&<\frac{5}{\sqrt 2}x^{2.5} \binom{x^2}{\frac{5}{\sqrt 2}x^{1.5}}<\frac{5}{\sqrt 2}x^{2.5} x^{5\sqrt{2} x^{1.5}}\\ \\
&< 5 x^{5\sqrt{2} x^{1.5}+2.5}< 5 x^{8 x^{1.5}}= e^{8 x^{1.5}\cdot \ln{x}+\ln{5}}< e^{9 x^{1.5}\cdot \ln{x}}
\end{align*}

% Details of how to get the above bound: plug w(G)=5, a=2f_1/n
% into the formula in Theorem 1 of~[GyarHS02]: 
% w(G) \geq 0.1 a^2 n^{-1}
   
% the last inequality holds because 1.5 < 0.5 x^{1.5} for all x >= 2 
        
%\todo[inline]{Eran: I changed a bit the computations and arguments above, please check. Could you make it look prettier?Danny: Checked, and made prettier.}
        
% so we have
% $$t(x)
% < 5 x^{8 x^{1.5}}
% = e^{8 x^{1.5}\cdot \ln{x}+\ln{5}}
% < e^{9 x^{1.5}\cdot \ln{x}}
% $$
% the last inequality ln 5 < x^{1.5} ln x holds for all x >= 2
and hence we obtain the claimed bound with constant $b=e^9$.
%    \end{enumerate}    
\end{proof}

\section{Proof of Theorem~\ref{cor:count-fns}(i): $4$-manifolds with factorially many different fns triangulations}\label{sec:symm}

The idea behind our construction is to realize geometrically the cycle structure of a permutation. We construct certain decorated building blocks to represent $1,\dots,k$, glue them in a row via the star connected sum operation, and then we further glue a decoration of the block corresponding to $i$ to a decoration of the block corresponding to $j$ whenever $\sigma(i)=j$, for $\sigma\in S_k$. We then show that the triangulations $T_{\sigma}$ and $T_{\pi}$ thus obtained are combinatorially distinct whenever $\sigma\neq\pi$, for permutations $\sigma,\pi\in S_k$.

\subsection*{Step 1: Building blocks and the skeleton $N_k$}
Let $M$ be our starting connected $4$-manifold, denote $\chi(M)=n$, and let $T_M$ be its fns triangulation, and $G:=G(T_M)$ the graph ($1$-skeleton) of $T_M$. Consider two vertices $u,v\in V(G)$ diametrically apart, and denote  $d(u,v)=\text{diam}(G)=:d$. 
We can have $d$ as large as we want, %either by taking a row (see below) of $M$'s as our starting $4$-manifold, or 
by letting $M=\mathbb{H}^4/K$ for an appropriate subgroup $K$ as explained in the Introduction, and in the resulted triangulation $T_M$ all vertex links are isomorphic to the boundary of the $600$-cell. How large we need $d$ to be will be determined along the proof.
Let $w$ be a third vertex (approximately) half way between $u$ and $v$. Explicitly: take $w$ on a path of length $d$ from $u$ to $v$ so that $d(u,w)=\lceil\frac{d}{2}\rceil$ and $d(v,w)=\lfloor\frac{d}{2}\rfloor$.

The \textbf{row} $R_a$ of length $a$ is the fns $4$-manifold~\footnote{We abuse notation here and denote the different copies of $T_M$ the same.}
$$
R_a:=\underbrace{T_M\#T_M\#\dots\#T_M}_{a\text{ times}}
$$
obtained by gluing $a$ disjoint copies of $T_M$ in a row, gluing the second copy along (the link of) its $v$ to the (link of) $v$ of the first copy, gluing the third copy along its $u$ to the $u$ of the second copy, and so on. A {\bf crucial} choice here is to perform each of the gluings with the identity map induced on the link of the vertices where the gluing occurs, that is, each neighbor $v'$ of $v$ of the first copy is identified with $v'$ of the second copy, and similarly with the other gluings. Slightly abusing notation, we denote the ``free" vertex (it is $u$ if $a$ is even, and it is $v$ when $a$ is odd) of the $a$-th copy by $v$, and thus we have {\bf the $u$ of $R_a$} and {\bf the $v$ of $R_a$}:
\begin{figure}[htbp]
\begin{center}
\includegraphics[scale=0.4]{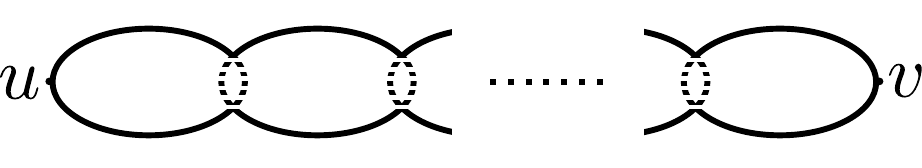}
%\caption{$-$}
\end{center}
\end{figure}

Let us note here that our crucial choice of gluing allows us to have a good estimate on the diameter of a row $R_a$:
%\todo[inline]{Eran: I corrected both the estimates (a-2 instead of a) and the argument for the lower bound. Please check.}
$$
2(d-1)+(a-2)(d-2) = d(u,v) \leq \diam(R_a) \leq 2(d+3)+(a-2)(d-2)
$$

The lower bound comes from considering a shortest path in $G$ from $u$ to $v$ (its length is $d$): $(u,u',\dots,v',v)$, and noting that concatenating the paths $(u,u',\dots, v')$ in copy $1$, $(v',\dots, u')$ in copy $2$, $(u',\dots, v')$ in copy $3$, and so on provides a shortest path of length $2(d-1)+(a-2)(d-2)$ from the $u$ of $R_a$ to the $v$ of $R_a$.

The upper bound comes from the observation that removing a vertex $x$ (and all edges incident to it) from $G$ can raise the diameter by at most $3$, since a path in $G$ of length $d$ that passed through the removed vertex $x$ can be replaced by a path of length $d-2+5=d+3$, where the $5$ comes from the fact that the diameter of the boundary of the $600$-cell is $5$.

Thus, we know the diameter of $R_a$ up to constant error, of $8$.
We will freely use the arguments given above for upper and lower bounds on lengths of various paths of interest in our constructions to come. 

A second building block employs the third distinguished vertex $w$:

The $4$-manifold $E$ is obtained by applying the star connected sum construction to $17$ copies of $M$ (of $T_M$ to be precise) according to the following scheme:
\begin{figure}[htbp]
\begin{center}
\includegraphics[scale=0.3]{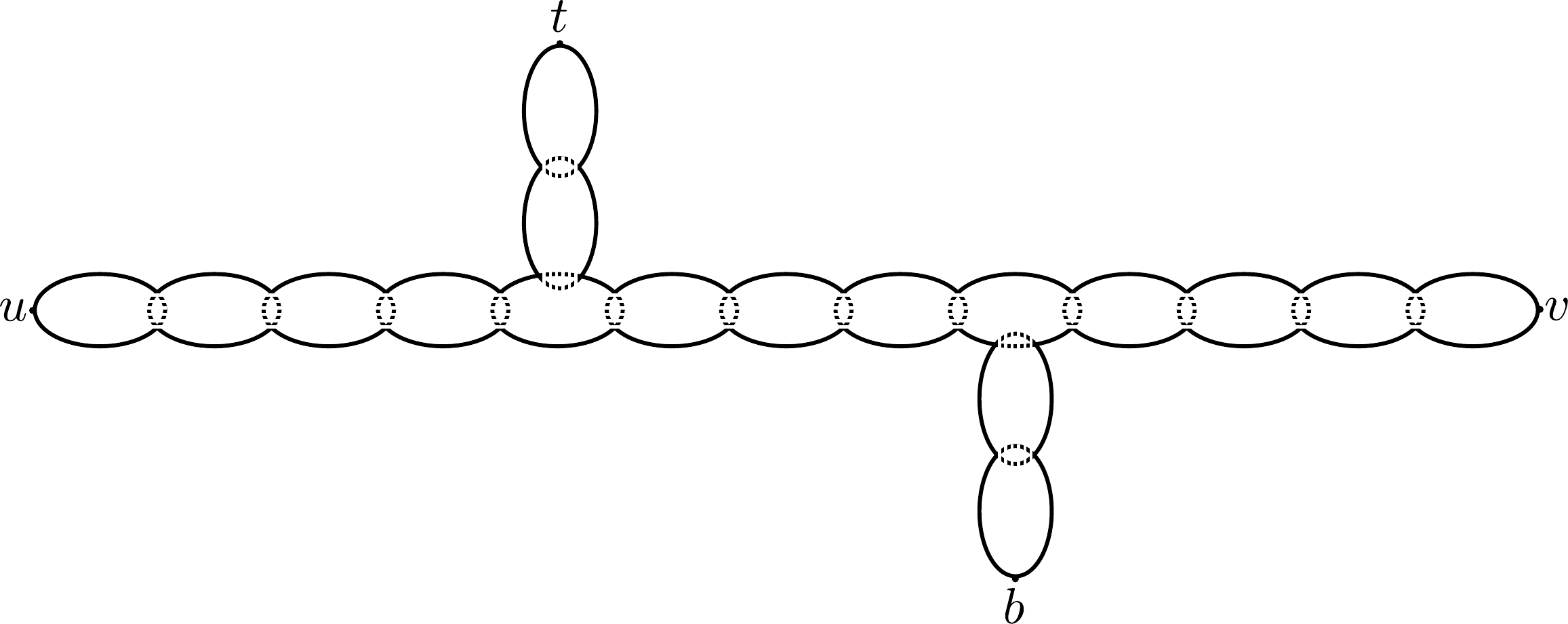}
%\caption{$-$}
\end{center}
\end{figure}

Here, the horizontal row is $R_{13}$; we glue the link of $w$ of the $5$th (resp. $9$th) copy of $M$ in $R_{13}$ to the link of $u$ of the copy of $M$ above (resp. below) it in the figure; and glue 
the link of $v$ of these later upper (resp. lower) copies of $M$ to the link of $v$ of the copy of $M$ above (resp. below) it. 
This triangulation $E$ has $4$ distinguished vertices: its $u,v,t$ and $b$. Here $u$ (resp. $v$) of $E$ is the $u$ (resp. $v$) of the row $R_{13}$, and $t$ (resp. $b$) of $E$ is the $u$ of the top (resp. bottom) copy of $M$. Employing similar arguments as given above for $\diam(R_a)$, we note here that
$$
\diam(E)\leq 2(d+3)+2(d+1)+9(d-2)=13d-10
$$
(this is almost the same computation as for $R_{13}$, except the fact that the distance between the links of $u$ and $v$ in the $5$th and $9$th copies of $M$ in the $R_{13}$ can grow by at most $3$ because of the removed vertex $w$).

%(this is almost the same computation as for $R_{13}$, except the fact that the diameter can grow by at most $6$ because of the two removed vertices in copies $5$ and $9$ of the $R_{13}$).

Let $E_1,\dots,E_k$ be $k$ copies of $E$, and denote the corresponding four  distinguished vertices of $E_i$ by $u_i,v_i,t_i,b_i$. We define a ``row" of star connected sums:
$$
N_k:=R_{15k}\#E_1\#\dots\#E_k\#R_{14k},
$$
where, as in the definition of a row $R_a$,  we glue the link of the ``free" vertex $v$ or $u$ of each but the last manifold in the row of star connected sums to the link of the same-labeled vertex of the next manifold in the row, via the identity map.
%Call the fns manifold $N_k$ \emph{the skeleton}.

Let $I_k$ be the induced subcomplex of $N_k$ obtained by removing the vertices of the left $R_{15k}$ and the right $R_{14k}$ which are not in $E_1$ nor in $E_k$. As $I_k$ is actually a row of $k$ copies of $E$, with $2$ vertices removed (the $u$ of the first $E$ and the $v$ of the last $E$), we have
\begin{align*}
\diam(I_k)&\leq \diam(E_1\#\dots\#E_k)+6\\
&\leq 2(13d-7)+(k-2)(13d-12)+6 = 13kd-12k+16
\end{align*}

%$(k-2)(12d-16)+2(12d-16+5+5) < 12dk$.
%Both these estimates, on the diameters of $E$ and $I_k$, follow from our choice of gluing maps, as explained before for $R_a$. 

%This choice of gluing maps guarantees that the diameter of the graph of $N_k$ is realized only by pairs of vertices where one vertex lies in the first copy of $M$ in the copy of $R_{2k}$ and the other vertex lies in the last copy of $M$ in the copy of $R_{k}$.

\subsection*{Step 2: Attaching handles to $N_k$ to get $M_k$}
Given a permutation $\sigma\in S_k$, we attach $k$ $1$-handles to $N_k$ and obtain a manifold $M_k$ with a fns triangulation $T_{\sigma}$. Explicitly, for each $1\leq i\leq k$ we glue a copy of $H$ (as in the figure below) along its $u$ to $t_i$ of $N_k$, and along its $v$ to $b_j$ of $N_k$, where $\sigma(i)=j$. We construct $H$ by gluing the link of $w$ of the bottom copy of $M$ to the link of $u$ of the copy of $M$ above it, and then gluing the link of $v$ of the latter to the link of $v$ of the the top copy of $M$. 

%Given a permutation $\sigma\in S_k$, we attach $k$ $1$-handles to $N_k$ and obtain a manifold $M_k$ with a fns triangulation $T_{\sigma}$. Explicitly, for each $1\leq i\leq k$ we glue to $N_k$ a copy of $H$ by two simultaneous star connected sums (or rather, a star connected sum followed by a star handle), where $H$ is as in the figure below:

\begin{figure}[htbp]
\begin{center}
\includegraphics[scale=0.4]{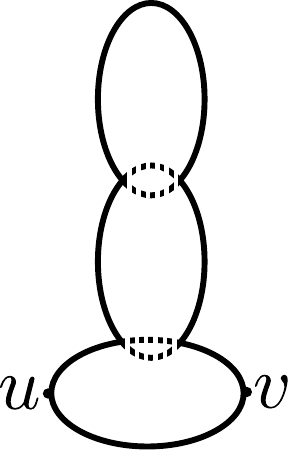}
%\caption{$-$}
\end{center}
\end{figure}
%and we glue $H$ along its $u$ to $t_i$, and along its $v$ to $b_j$, where $\sigma(i)=j$.
%(Here, in $H$ we glued the link of $w$ of the bottom copy of $M$ to the link of $u$ of the copy of $M$ above it, and glued the link of $v$ of the latter to the link of $v$ of the the top copy of $M$.)

By Lemmas~\ref{lem:star-connected-sum} and~\ref{lem:handle-connected-sum}, all the fns complexes $T_\sigma$ thus obtained, for $\sigma\in S_k$, triangulate the same topological manifold, denoted by $M_k$.

\subsection*{Step 3: Reconstructing $\sigma$ from $T_{\sigma}$}
To show that all $k!$ triangulations of $M_k$ are distinct, namely pairwise non-isomorphic, we assume that we are given a triangulation $T\in\set{T_{\sigma}}{\sigma\in S_k}$, and we determine the unique $\sigma\in S_k$ for which $T=T_{\sigma}$.

Let $G=G(T)$ be the graph of $T$, and let $s,e\in V(G)$ be two vertices diametrically apart in $G$. 

%\todo[inline]{Eran: 1. the "Clearly" part might be false, e.g. for $\sigma=(1k)$. We need the row parts longer than the middle part in $N_k$. Middle part has length at most $12dk$, so take say two $R_{13dk}$ for def of $N_k$. I corrected below. 2. I changed above def of $E$ and $N_k$ to get good diameter estimates for the claim claim on s,e to hold.
%Danny: I changed the analysis of the diameters of $E$ and of $I_k$ in step 1 above. This also effected the analysis below. Please verify these changes and remove this todo box.}

%Clearly, 
%$s$ must lie in the first copy of $M$  in the row $R_{2k}$ and $e$ must lie in the last copy of $M$ in the row $R_k$.

We argue that one of $s,e$ (say $s$) must lie in the first copy of $M$ of the left $R_{15k}$, and the other ($e$ then) must lie in the last copy of $M$ of the right $R_{14k}$. Indeed, note that for every permutation $\sigma\in S_k$, by removing from $T_\sigma$ the vertices of the left $R_{15k}$ and the vertices of the right $R_{14k}$ that do not lie in $E_1$ nor in $E_k$, we obtain a subcomplex $Q_k$. The diameter of $Q_k$ does not exceed the diameter of $I_k$ (introduced towards the end of Step 1) because $Q_k$ is basically $I_k$ with some handles added to it that can act like shortcuts between far away vertices. So the diameter of $Q_k$ is at most $13dk$, smaller than the diameter of $R_{14k}$ with a vertex removed. Thus, $s$ and $e$ that realize the diameter of $G$ must be located as claimed.\footnote{The permutations that make $s$ and $e$ as above the closest, are those mapping $1$ to $k$, and their inverses.}

%A key ingredient for the reconstruction is the determination of the \textbf{junctions}, defined as the copies of $M$ that have been glued on their $u,v$ and $w$. 
A key ingredient for the reconstruction is the determination of the junctions. A copy of $M$ is called a \textbf{junction} if it has been glued on all of its three distinguished vertices $u,v$ and $w$. A junction in a copy of $H$ is called a \textbf{handle-junction}. To identify the junctions of $T$, consider for every vertex $x\in V(G)$ the \textbf{sphere of radius  $\left\lceil\frac{3}{2}d\right\rceil$ centered at $x$}:

$$
\s_x=\set{y\in V(G)}{d(y,x)=\left\lceil\frac{3}{2}d\right\rceil}.
$$

%$$
%A_x=\set{y\in V(G)}{\frac{3}{2}d\leq d(y,x)\leq 2d-4}.
%$$
%In the following Lemma~\ref{lem:3components} we will use the term \textbf{cluster} to mean a subset of vertices in which any two vertices are at distance at most $d+6$ (in $G$).
Let $B_l(x)$ denote the set of vertices in $G$ of graph distance at most $l$ from $x$. Note that $\s_x$ is the boundary of $B_{\left\lceil\frac{3}{2}d\right\rceil}(x)$.

%\todo[inline]{We need to change from annulus to a sphere of points exactly 3d/2 away from x and change computations accordingly.}

%\todo[inline]{Eran: I rephrased the Lemma below, please check, in particular I changed "clusters" to "connected components", are the $C_i$ indeed connected? . Also, I changed vertex notations from u,v to x,y to avoid confusion.

%Danny: I don't see why the induced graph on $C_i$ must be connected, that's why I initially called them clusters. Other changes are fine.}

\begin{lemma}\label{lem:3components}
Let $d\ge 77$.
\begin{enumerate}
\item If $x$ is in a junction then $\s_x$ is partitioned into three sets, called \textbf{clusters}:
$$\s_x=C_1\cup C_2\cup C_3$$ where two vertices from the same cluster have distance at most $d+18$, and two vertices from different clusters have distance at least $\frac{3}{2}d-20$ (which is larger then $d+18$ by our assumption on $d$).
\item Assume $x$ is in a junction $J$. Then $J$ is a handle-junction if and only if $G\setminus B_{\left\lceil\frac{3}{2}d\right\rceil}(x)$ is disconnected with one small component (of at most $2|V(T_M)|$ vertices), and a second component containing all other vertices. 
\end{enumerate}
\end{lemma}

\begin{proof}
Let us consider the $2d-4$ neighborhood $N$ of a junction $J$, namely $N$ is the induced subgraph of $G$ consisting of all vertices of graph distance at most $2d-4$ from some vertex of $J$. Denote the copies of $M$, and the separating $600$-cells boundaries where their gluings (i.e. star connected sum) occurred, as in the figure below:

\begin{figure}[htbp]
\begin{center}
\includegraphics[scale=0.5]{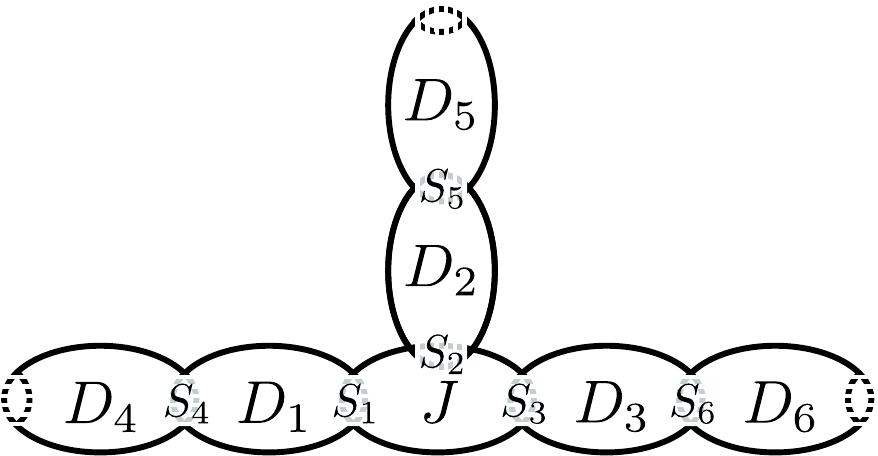}
%\caption{$-$}
\end{center}
\end{figure}

Since $x\in J$ is at distance at most $d+9$ from any vertex of $S_1\cup S_2\cup S_3$ (recall this distance estimate comes from the fact that the diameter of the $600$-cell graph is $5$), 
each vertex in $\s_x$ must lie outside $J$, at least $\frac{3}{2}d-(d+9)=\frac{1}{2}d-9$ away from $J$,
%$S_1\cup S_2\cup S_3$
 for $d\ge 19$.
 On the other hand, any vertex of $N$ not in $J$ 
% and of distance at most $2d-4$ from  
%$J$ 
must belong to $D_1\cup\dots\cup D_6$, hence the sphere $\s_x$ consists of the three light blue arcs in the figure below:
\begin{figure}[htbp]
\begin{center}
\includegraphics[scale=0.5]{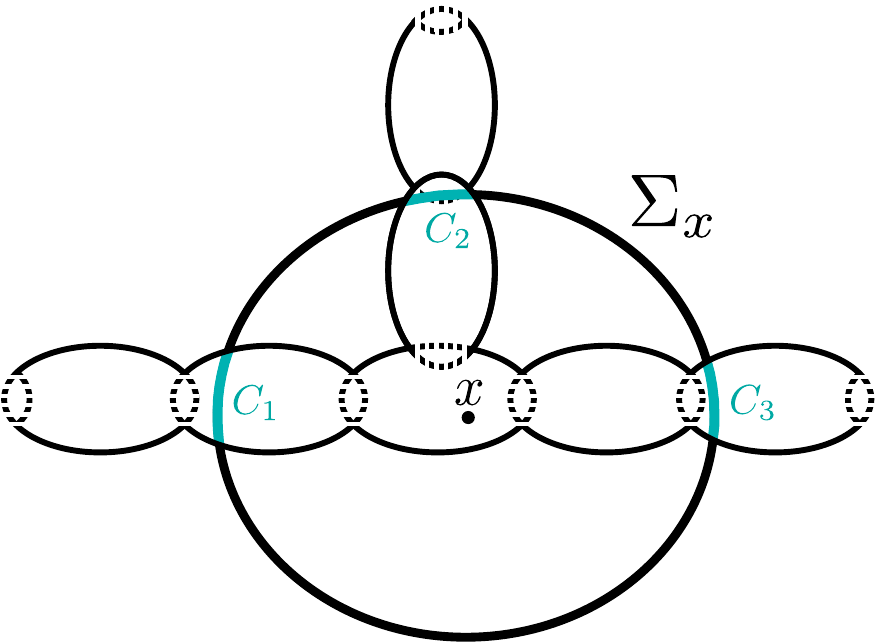}
%\caption{$-$}
\end{center}
\end{figure}

which satisfy the conditions in item (1) for clusters:
\begin{align*}
C_1&=\s_x\cap\left(D_1\cup D_4\right)\\
C_2&=\s_x\cap\left(D_2\cup D_5\right)\\
C_3&=\s_x\cap\left(D_3\cup D_6\right)
\end{align*}

Indeed, consider two vertices $p,q$ in $C_1$. If both belong to $D_1$ or both belong to $D_4$, then they are at most $d+6$ away from each other. 

Suppose $p\in D_1$ and $q\in D_4$. First we show that because $p\in D_1$, $x$ cannot be too close to $S_1$. Since every shortest path from $p$ to $x$ must pass through $S_1$, we have $\left\lceil\frac{3}{2}d\right\rceil=d(p,x)\le d(p,S_1)+\diam(S_1)+d(S_1,x)$. Using $d(p,S_1)\le d+3$ and $\diam(S_1)=5$, we get $\left\lceil\frac{d}{2}\right\rceil-8\le d(x,S_1)$. 

Now since every shortest path from $x$ to $q$ must pass through $S_1$ and $S_4$, we have $\left\lceil\frac{3}{2}d\right\rceil=d(q,x)\ge d(x,S_1)+d(S_1,S_4)+d(S_4,q)$. Using $d(S_1,S_4)=d-2$ and $d(x,S_1)\ge \left\lceil\frac{d}{2}\right\rceil-8$, we get $d(S_4,q)\le 10$. 

Finally, since every shortest path from $p$ to $q$ must pass through $S_4$, we have $d(p,q)\le d(p,S_4)+\diam(S_4)+d(S_4,q)$. Now, using $d(p,S_4)\le d+3$, $d(S_4,q)\le 10$, and $\diam(S_4)=5$, we get $d(p,q)\le d+18$.

%If $p\in D_1$ and $q\in D_4$ then $p$ is at distance at most $\frac{1}{2}d$ from $S_4$ and as $d(p,x)\ge \frac{3}{2}d$ also $x$ is at least $\frac{1}{2}d$ away from $S_1$; as $d(q,x)<2d$ also $q$ lies at distance at most $\frac{1}{2}d$ from 
%some vertex on $S_4$ and so  $d(p,q)\le \frac{1}{2}d+\frac{1}{2}d+5=d+5$.

The same argument works for $C_2$ and $C_3$.

Next, consider two vertices from different clusters, say $u_1\in C_1$ and $u_2\in C_2$. Since a shortest path from $u_1$ to $u_2$ must pass through $S_1$ and $S_2$ we have:
\begin{align*}
d(u_1,u_2)&\geq d(u_1,S_1)+d(S_1,S_2)+d(S_2,u_2)\\
%d(u_1,u_2)&\geq d(u_1,S_1)+d(S_1,S_2)+d(S_1,u_2)\\
&\geq \left(\frac{1}{2}d-9\right)+\left(\frac{1}{2}d-2\right)+\left(\frac{1}{2}d-9\right)=\frac{3}{2}d-20.
\end{align*}
Similarly, the same estimate holds for the other two choices of a pair of clusters $C_i$.
This completes the proof of $(1)$.

To prove $(2)$ notice that if $x$ is in a handle-junction $J$ then $G\setminus B_{\left\lceil\frac{3}{2}d\right\rceil}(x)$ is indeed composed of two connected components: the top part of an $H$, and the rest of the graph $G$. While, if $J$ is a non-handle-junction then either $G\setminus B_{\left\lceil\frac{3}{2}d\right\rceil}(x)$ is connected, or $G\setminus B_{\left\lceil\frac{3}{2}d\right\rceil}(x)$ has two ``large" components: one contains the left $R_{15k}$, while the other contains the right $R_{14k}$; each has more than $2|V(T_M)|$ vertices.
\end{proof}

%\todo[inline]{Eran: Remark below was unclear, I rephrased, please check. Danny: Checked.}

\begin{remark}\label{rem:middleVertex}
The conclusion in item (1) of Lemma~\ref{lem:3components} could hold even for vertices not in a junction but only close to a junction (for instance, vertices in $D_1$, close to $S_1$). Call these \textbf{junction vertices}. We can still identify a vertex that is actually \textbf{in} a junction by considering a maximal sequence of consecutive junction vertices along an induced path in $G$ passing from left to right (i.e., through the sequence of components $D_4,D_1,J,D_3,D_6$) and taking the middle vertex in this sequence.

%The conclusion in item (1) of the lemma could hold for vertices not in a junction, but only close to a junction (for instance, vertices in $D_1$, close to $S_1$). Call these \textbf{junction vertices}. Yet, we can identify a vertex \textbf{in} the junction by considering a maximal sequence of consecutive junction vertices along an induced path in $G$ passing from left to right -- through the sequence of components $D_4,D_1,J,D_3,D_6$ -- and taking the middle vertex in this sequence.
\end{remark}

Note that for every path $P$ from $s$ to $e$ that intersects each $E_i$ only in its row part, each vertex $x\in P$ satisfies that $G\setminus B_{\left\lceil\frac{3}{2}d\right\rceil}(x)$ is either connected or has two large connected components, each of size larger than $2|V(T_M)|$. Furthermore, every shortest path from $s$ to $e$ among those \emph{not} intersecting any copy of $H$ is of the same form as $P$ above.  
Thus, using Lemma~\ref{lem:3components}, we may identify a shortest path $\pi$ from $s$ to $e$ among those that do not pass through handle-junctions:
%, and we take a shortest such path 
$$\pi=\left(s=p_1,p_2,\dots,p_r=e\right).$$ 
Walking along $\pi$, and considering the sequences of spheres $\left(\s_{p_i}\right)_{i=1}^r$ we will encounter a  subsequence of consecutive junction vertices (maximal w.r.t. inclusion) after about $15k(d-2)$ steps 
(for $k$ large, up to additive error depending only on $d$).
Its middle vertex  $u_1$ is in a junction $J_{1,1}$ of $E_1$; see Remark~\ref{rem:middleVertex}. 
%We denote by $u_1$ the mid vertex of the subsequence. 
Continuing the sequence of spheres we will have a short junction-free subsequence, followed by a second (maximal) subsequence of consecutive junction vertices --  its middle vertex  $v_1$ is in a junction $J_{2,1}$ of $E_1$. 

%corresponding to $J_2$ of $E_1$. We denote by $v_1$ the mid vertex of the subsequence.

Proceeding in this manner, we distinguish the vertices 
$$u_1,v_1,u_2,v_2,\dots,u_k,v_k$$ 
where $u_i$ lies in junction $J_{1,i}$ of $E_i$, and $v_i$ lies in junction $J_{2,i}$ of $E_i$, for each $1\leq i\leq k$.

We can now recover $\sigma$:

%\todo[inline]{Eran: I changed both phrasing and pf below, please check.}

\begin{lemma} Let $d\ge 
77$.
%68$.
For each $1\leq i,j\leq k$ we have:
\begin{enumerate}
\item If  $\sigma(i)=j$, then $d(u_i,v_j)\leq 7d+51$.

%$d(u_i,v_j)<8d-16$ for some $j\neq i$ then $\sigma(i)=j$, otherwise
\item If  $\sigma(i)\neq j\neq i$, then $d(u_i,v_j)>8d-16$. 
%for all $j\neq i$ then $\sigma(i)=i$.
\end{enumerate}
Thus, for all $1\leq i\leq k$, if there exists $j\neq i$ such that $d(u_i,v_j)\leq 7d+51$ then this $j$ is unique and 
$\sigma(i)=j$, else $\sigma(i)=i$.
\end{lemma}
\begin{proof}
As discussed above, for $d\ge 77$ the junction vertices $u_i\in J_{1,i}$ and $v_j\in J_{2,j}$ are well defined for $1\le i,j\le k$. 
If $\sigma(i)=j$, then consider a shortest path from $u_i$ to $v_j$ among those that pass through the handle $H$ connecting the copy of $M$ containing $t_i$ to the copy of $M$ containing $b_j$. Such a path starts with $u_i$,  crosses the junction $J_{1,i}$ of $E_i$ along at most $d+9$ edges, then crosses the next two copies of $M$ reaching the handle $H$ along at most $2(d+6)$ edges, then crosses $H$ along at most $d+9$ edges, then crosses the next two copies of $M$ reaching the junction $J_{2,j}$ of $E_j$ along at most $2(d+6)$ edges, and continues in this $J_{2,j}$ to $v_j$ along at most $d+9$ edges. Altogether this path has length at most $7d+51$, proving item (1).

Next, assume  $\sigma(i)\neq j\neq i$.
Then, a shortest path from $u_i$ to $v_j$ either crosses at least eight consecutive non-junction copies of $M$ in a row -- passing from some $E_k$ to $E_{k\pm 1}$ -- thus such path has length at least $8(d-2)$, or it passes  through at least two handles, yielding a path of length at least $13(d-2)$. This proves item (2).

% \todo[inline]{Danny: (a) I think that the last argument "a shortest path from $u_i$ to $v_j$ must cross at least eight consecutive non-junction copies of $M$ in a row -- at least four from $E_i$ and at least four from $E_j$" is false. Such a shortest path can go through handle(s), not having at least $8$ consecutive non-junction copies as claimed. 

% (b) I am still not comfortable with this "opposite" formulation: "If $\sigma(i)=j$, then..." and "assume  $\sigma(i)\neq j\neq i$...". I think this is misleading.

% Eran: Thanks, I incorporated this case, which leads to an even longer distance, please check.}

%As $d\ge 68$,
As $d\ge 77$, then $8(d-2)>7d+51$, and the recovery of the unique permutation $\sigma$ for which $T=T_\sigma$ is complete. 
%%%
% \iffalse
% If $\sigma(i)\neq j$ then there is no handle connecting $t_i$ to $b_j$, and thus a shortest path from $u_i$ to $v_j$ must start from $u_i$ and go to some junction on $\pi$.
% \todo[inline]{Eran: the argument below is false: "The closest junction to $u_i$ on $\pi$ is where $v_i$ lies," so I changed it.}
% The closest junction to $u_i$ on $\pi$ is where $v_i$ lies, and it is at least $3(d-2)$ away. Continuing from $v_i$ we must go at least $5(d-2)$ more steps to get to one of our distinguished vertices, so we have $d(u_i,v_j)>8d-16$ in this case. The second closest junction to $u_i$ is $v_{i-1}$ which is at least $8d-16$ away as well. This proves $(1)$. 

% Furthermore, if $d(u_i,v_j)>8d-16$ for all $j\neq i$, then $t_i$ had not been connected by a handle to any $b_j$, with $j\neq i$, thus the unique handle that was glued to $t_i$ must have been glued to $b_i$ on its other end, which means that $\sigma(i)=i$, as claimed in $(2)$.
% \fi
\end{proof}

To finish the proof of Theorem~\ref{cor:count-fns}(i) let us observe that the Euler characteristic of $M_k$ is indeed linear in $k$: denote $n=\chi(M)$. First, notice that attaching a handle $H$ contributes $\chi(H)-4=3n-8$ to $\chi$, by Lemmas~\ref{lem:star-connected-sum}(ii) and~\ref{lem:handle-connected-sum}(ii). Thus:
$$
\chi(M_k)=\chi(N_k)+k(3n-8)
$$
Now, we compute $\chi(N_k)$ using Lemmas~\ref{lem:star-connected-sum}(ii) and~\ref{lem:handle-connected-sum}(ii):
\begin{align*}
\chi(N_k)&=\chi(R_{15k}\#E_1\#\dots\#E_k\#R_{14k})\\
&=\chi(R_{15k})+k\chi(E)+\chi(R_{14k})-2(k+1)\\
&=\left(15kn-2(15k-1)\right)+k\left(17n-2\cdot 16\right)+\left(14kn-2(14k-1)\right)-2(k+1)\\
&=(46n-92)k+2
\end{align*}
so we conclude
$$
\chi(M_k)=(46n-92)k+2+k(3n-8)=(49n-100)k+2=\Theta(k)
$$
as claimed.

\section{Concluding remarks}\label{sec:Open}
In view of Corollary~\ref{cor:highEvenXhi}, we ask:
\begin{prob}
Is there a $4$-manifold of \emph{odd} Euler characteristic that admits a fns triangulation?
\end{prob}
If the answer is Yes, with a construction admitting a vertex whose link is isomorphic to the boundary of the 600-cell, then gluing it to the fns manifolds $G_s$ of Corollary~\ref{cor:highEvenXhi} via a star connected sum along the link of such vertex, would yield connected fns $4$-manifolds realizing every large enough integer as their Euler characteristic. 

In view of Theorem~\ref{cor:count-fns}, we ask:
\begin{prob}
Is the number of fns triangulations of a $4$-manifold $M_i$ super-factorial in $\chi(M_i)$ for a suitable sequence of manifolds $(M_i)_i$?
\end{prob}
Similarly, can the upper bound of Theorem~\ref{cor:count-fns} on $t(M_i)$ be improved?
Is $t(x)$ of larger order of magnitute than any $t(M)$ where $\chi(M)=x$ and $x$ tends to infinity?

Next, we consider the piecewise linear structure of our constructed manifolds:
\begin{prob}
Are the different combinatorial triangulations $T_\sigma$ of the $4$-manifold $M_k$ we get in Theorem~\ref{cor:count-fns} PL-homeomorphic?
\end{prob}

\bibliographystyle{alpha}
\bibliography{star}

\end{document}